\documentclass[fontsize=12pt]{amsart}

\usepackage[object=vectorian]{pgfornament}
\usepackage{amsmath}
\usepackage{amsthm}
\usepackage{amssymb}
\usepackage{mathtools}
\usepackage{mathrsfs}
\usepackage{physics}
\usepackage[margin=1in]{geometry}
\usepackage{scrextend}
\usepackage{tikz-cd}
\usepackage[shortlabels]{enumitem}
\usepackage{colonequals}
\usepackage[hyperfootnotes]{hyperref}
\usepackage{comment}

\usepackage[notcolon, notquote]{hanging}

\usepackage{fancyhdr} 
\renewcommand{\P}{\mathbb{P}}

\newcommand{\C}{\mathbb{C}}
\newcommand{\cC}{\mathcal{C}}
\newcommand{\Z}{\mathbb{Z}}
\newcommand{\Q}{\mathbb{Q}}

\newcommand{\I}{\mathcal{I}}
\renewcommand{\O}{\mathcal{O}}
\newcommand{\X}{\mathcal{X}}
\newcommand{\E}{\mathcal{E}}

\DeclareMathOperator{\Bl}{Bl}

\DeclareMathOperator{\NS}{NS}
\DeclareMathOperator{\gon}{gon}
\DeclareMathOperator{\cg}{cov.gon}
\DeclareMathOperator{\irr}{irr}
\DeclareMathOperator{\mult}{mult}

\DeclareMathOperator{\coker}{coker}

\newtheorem{thm}{Theorem}[section]

\newtheorem{lemma}[thm]{Lemma}
\newtheorem{proposition}[thm]{Proposition}
\newtheorem{conj}[thm]{Conjecture}

\newtheorem{Lthm}{Theorem}

\theoremstyle{definition}
\newtheorem{remark}[thm]{Remark}

\newtheorem{example}[thm]{Example}
\newtheorem{definition}[thm]{Definition}

\newtheorem{setup}[thm]{Set-up}

\newtheorem*{acknowledgements}{Acknowledgements}
\newtheorem*{claim}{Claim}

\begin{document}

\title{Multiplicative bounds for measures of irrationality on complete intersections}
\author{Nathan Chen}

\thanks{The author's research was partially supported by an NSF postdoctoral fellowship, DMS-2103099.}

\maketitle

\allowdisplaybreaks

\thispagestyle{empty}

The purpose of this paper is to show that measures of irrationality on very general codimension two complete intersections and very general complete intersection surfaces are multiplicative in the degrees of the defining equations.

In recent years, there has been growing interest in studying measures of irrationality for projective varieties. As a higher dimensional generalization of gonality, these birational invariants quantify in various ways how far a given variety $X$ is from being rational. We will focus primarily on two of these measures, the \textit{degree of irrationality} and the \textit{covering gonality}, which are defined as follows:
\[ \irr(X) = \min \left\{ \delta > 0 \ \middle| \ \exists \text{ degree $\delta$ rational covering } X \dashrightarrow \P^{\dim X} \right\}; \]
\[ \cg(X) = \min \left\{ c > 0 \ \middle| \ \let\scriptstyle\textstyle \substack{
				\text{Given a general point } p \in X, \ \exists \text{ irreducible} \\[0.1cm] \text{curve } C \subseteq X \text{ through $p$ with gonality } c} \right\}. \]

In the case of hypersurfaces of large degree, the situation is now fairly well understood. Specifically, let $X \subset \P^{n+1}_{\C}$ be a smooth hypersurface of degree $d$ and dimension $n \geq 2$. If $X$ is very general of degree $d \geq 2n+1$, then Bastianelli, De Poi, Ein, Lazarsfeld, and Ullery \cite{BDELU17} have shown that $\irr(X) = d-1$. In the same paper, the authors proved that if $X$ is arbitrary and $d \geq n+2$, then $\cg(X) \geq d-n$. The central theme of \cite{BDELU17} is that the positivity properties of the canonical bundle yield lower bounds for measures of irrationality. Using different methods, Bastianelli, Ciliberto, Flamini, and Supino \cite{BCFS18} later computed $\cg(X) \approx d - 2 \sqrt{n}$ for very general hypersurfaces of degree $d \gg 0$.

A logical next step is to investigate the behavior of these invariants for complete intersection varieties in projective space. For complete intersections over $\C$, the same techniques in \cite{BDELU17} yield lower bounds for the covering gonality which are additive in the degrees of the defining equations. Recently, Smith \cite{Smith20} has extended these results about the covering gonality of complete intersections to positive characteristic.

However, it has been conjectured \cite[Problem 4.1]{BDELU17} that measures of irrationality on complete intersections should be \textit{multiplicative} in the degrees. As evidence, Lazarsfeld \cite[Exercise 4.12]{Lazarsfeld97} had established that the gonality of a smooth complete intersection curve $C \subset \P_{\C}^{e+1}$ of type $(a_{1}, a_{2}, \ldots, a_{e})$ with $2 \leq a_{1} \leq \cdots \leq a_{e}$ is bounded from below by $\gon(C) \geq (a_{1}-1) a_{2} \cdots a_{e}$. Further refinements due to Hotchkiss, Lau, and Ullery \cite{HLU20} show that when $4 \leq a_{1} < a_{2} \leq \cdots \leq a_{e}$ holds, the gonality of the curve $C$ is realized by projection from a suitable linear subspace. In higher dimensions, Stapleton \cite{Stapleton17} used results about Seshadri constants on hypersurfaces which were due to Ito \cite{Ito14} to give bounds for the covering gonality of codimension two complete intersections that were stronger than additive. Later, Stapleton and Ullery \cite{SU20} computed the degree of irrationality for codimension two complete intersections of type $(2, d)$ and $(3, d)$.

Our first result shows that the covering gonality of very general codimension two complete intersections is multiplicative:

\begin{Lthm}\label{thm:cgcodim2}
Let $X \subset \P^{n+2}_{\C}$ be a very general smooth complete intersection of type $(a, b)$ and dimension $n \geq 2$. If $a, b \geq 18n/7$, then
\[ \cg(X) \geq \frac{2}{3 (n+1)^{2}} \cdot ab. \]
\end{Lthm}

\noindent Since in general $\irr(X) \geq \cg(X)$, we obtain the same inequality for the degree of irrationality.

Our second theorem gives multiplicative bounds for complete intersection surfaces:

\begin{Lthm}\label{thm:cgCIsurfaces}
Let $X \subset \P_{\C}^{e+2}$ be a very general smooth complete intersection surface of type $(d_{1}, \ldots, d_{e})$. There exist positive constants $A = A(e)$ and $B = B(e)$ such that if $d_{i} \geq A$ for all $1 \leq i \leq e$, then
\[ \cg(X) \geq B \cdot d_{1} \cdots d_{e}. \]
\end{Lthm}

\noindent See the end of \S 1, Proposition~\ref{degreeCurves}, and Remark~\ref{remark:primes} for explicit constants.

In \S 1, we will present a reduction step which first appeared in the work of Stapleton \cite[\S 5.2]{Stapleton17}. The theorems will reduce to showing that certain families of line bundles on the blow-ups of complete intersection varieties are big and nef, which can be thought of as multi-point Seshadri constants. See Theorems~\ref{thm:codim2nef} and \ref{thm:dim2nef} for the explicit statements and details. In \S 2, we will collect several tools that will be used to control the numerical invariants of curves on complete intersections. In \S 3, we will prove Theorem~\ref{thm:codim2nef}. The proof of Theorem~\ref{thm:dim2nef} will take up most of \S 4. Throughout the paper, we work over $\C$.

\begin{acknowledgements}
I am grateful to Robert Lazarsfeld for many valuable discussions and for suggesting the approach to Proposition~\ref{prop:multiplicityDeltaInvariant}. I would also like to thank Olivier Martin, Mihnea Popa, and David Stapleton for giving valuable feedback about an earlier draft of the paper and thank Aaron Landesman for helpful discussions. A version of Theorem~\ref{thm:cgcodim2} was part of the author's Ph.D. thesis at Stony Brook University.
\end{acknowledgements}

%
%

\section{Reduction step}

In this section, we will show how Theorems~\ref{thm:cgcodim2} and \ref{thm:cgCIsurfaces} follow from the nefness of certain families of line bundles. Our starting point is the following result \cite[Theorem 1.10]{BDELU17}, which says that positivity properties of the canonical bundle lead to lower bounds on the covering gonality (see also \cite[Remark 5.14]{Stapleton17}):

\begin{proposition}\label{prop:covgonX}
Let $X$ be a smooth projective variety and suppose that there exists an integer $r$ such that the canonical bundle $K_{X}$ separates $r$ points on an open set. Then
\[ \cg(X) \geq r+1. \]
\end{proposition}
We will apply this proposition as follows. Consider the inclusions $X \subset Y \subset \P^{n+e}$, where
\begin{list}{$\diamond$}{}
\item $Y$ is a complete intersection of dimension $n+1$ and type $(a_{1}, \ldots, a_{e-1})$, and
\item $X \in \abs{\O_{Y}(a_{e})}$ is a complete intersection of dimension $n$ and type $(a_{1}, \ldots, a_{e})$.
\end{list}
We would like to show that $K_{X}$ separates $r$ points on an open set, for suitable $r$. By passing to a complete intersection $Y$ of larger dimension, we can take advantage of adjunction:

\begin{proposition}\label{prop:cg+nef}
Suppose that there exists a positive integer $r$ such that on the blow-up $\mu \colon \widetilde{Y} \rightarrow Y$ along any distinct points $p_{1}, \ldots, p_{r} \in X$ with exceptional divisors $E_{1}, \ldots, E_{r}$, the line bundle
\[ L \colonequals \mu^{\ast}\O_{Y}(X) - \sum_{i=1}^{r} (n+1)E_{i} \]
is nef and big. Then $\cg(X) \geq r+1$.
\end{proposition}

\begin{proof}[Proof of Proposition~\ref{prop:cg+nef}]
By Kawamata-Viehweg vanishing for big and nef line bundles,
\begin{align*}
H^{1} \left(Y, (K_{Y} + \O_{Y}(X)) \otimes \I_{\{ p_{1}, \ldots, p_{r} \}} \right) &= H^{1} \left( \widetilde{Y}, \mu^{\ast}(K_{Y} + \O_{Y}(X)) - \sum_{i=1}^{r} E_{i} \right) \\
&= H^{1}(\widetilde{Y}, K_{\widetilde{Y}} + L) = 0.
\end{align*}
Here we use the fact that $K_{\widetilde{Y}} \cong \mu^{\ast}K_{Y} + n E$. The vanishing above gives a surjection
\[ H^{0}(Y, K_{Y} + \O_{Y}(X)) \twoheadrightarrow H^{0}\left( Y, (K_{Y} + \O_{Y}(X)) \otimes \O_{\{ p_{1}, \ldots, p_{r} \}} \right). \]
In other words, sections of the adjoint bundle $K_{Y} + \O_{Y}(X)$ separate any finite set of $r$ distinct points in $X$. By the adjunction formula, $K_{X} \cong (K_{Y} + \O_{Y}(X)) \big|_{X}$ and hence sections of $K_{X}$ separate any finite set of $r$ distinct points in $X$. Proposition~\ref{prop:covgonX} implies that $\cg(X) \geq r+1$.
\end{proof}
\noindent In practice, once we prove nefness of $L$, it will follow numerically that $L$ is big \cite[Theorem 2.2.16]{Lazarsfeld04}. By Proposition~\ref{prop:cg+nef}, it suffices to prove nefness of families of line bundles in two different settings:

\begin{setup}[Codimension two complete intersections]\label{setup:codim2nef}
Let $n \geq 2$ be arbitrary and set $e = 2$. Then $X \subset Y \subset \P^{n+2}$, where $Y \in \abs{\O_{\P^{n+2}}(a)}$ is a very general hypersurface and $X \in \abs{\O_{Y}(b)}$ is a very general complete intersection of dimension $n$ such that $b \geq a \geq 18n/7$. With this, we will show:
\end{setup}

\begin{thm}\label{thm:codim2nef}
Consider Set-up~\ref{setup:codim2nef} and fix an integer
\[ r \leq \frac{2}{3(n+1)^{2}} ab. \]
For any set of distinct points $p_{1}, \ldots, p_{r} \in Y$, if we let $\mu \colon \widetilde{Y} \rightarrow Y$ denote the blow-up of $Y$ at these points with exceptional divisor $E_{i}$ over $p_{i}$ and set $H = \mu^{\ast}\O_{Y}(1)$, then the following divisor on $\widetilde{Y}$ is nef:
\[ bH - \sum_{i=1}^{r} (n+1)E_{i}. \]
\end{thm}

\noindent Granting Theorem~\ref{thm:codim2nef} for now, we will first show how it implies Theorem~\ref{thm:cgcodim2}.

\begin{proof}[Proof of Theorem~\ref{thm:cgcodim2}]

In the setting of Set-up~\ref{setup:codim2nef}, fix $n \geq 2$, choose $b, a \geq 18n/7$, and set $r = \left\lfloor \frac{2}{3(n+1)^{2}} \cdot ab \right\rfloor$. Theorem~\ref{thm:codim2nef} shows that for any tuple of $r$ distinct points $p_{1}, \ldots, p_{r} \in X$, the divisor
\[ L \colonequals bH - \sum_{i=1}^{r} (n+1)E_{i} \]
on the blow-up $\mu \colon \widetilde{Y} \rightarrow Y$ is nef. It is straightforward to check that $(L^{n+1}) > 0$ on $\widetilde{Y}$:
\[ \left( bH - \sum_{i=1}^{r} (n+1)E_{i} \right)^{n+1} = a b^{n+1} - r \cdot (n+1)^{n+1} \geq a b^{n+1} - \frac{2}{3} (n+1)^{n-1} ab > 0 \]
holds as long as $b \geq a \geq n+1$, so $L$ is also big. By Proposition~\ref{prop:cg+nef},
\[ \cg(X) \geq r+1 \geq \frac{2}{3(n+1)^{2}} \cdot ab. \qedhere \]
\end{proof}

\begin{setup}[Complete intersection surfaces]\label{setup:ciSurfacenef}
Set $n = 2$ and let $e \geq 2$ be arbitrary. We will first consider a very general smooth complete intersection surface
\[ X \subset Y \subset \P^{e+2} \qquad \text{of type} \qquad (a_{1}, \ldots, a_{e}), \]
where the $a_{i}$ are integers of a special form. More precisely, assume that $3e \leq a_{1} \leq \cdots \leq a_{e}$ and
\[ a_{i} \colonequals (e+1)! \cdot q_{i} \qquad \text{for } i = 1, \ldots, e-1, \]
where the $q_{i}$ are positive integers which are pairwise coprime. Let $Y$ be the smooth complete intersection threefold of type $(a_{1}, \ldots, a_{e-1})$ containing $X$. We will prove:
\begin{thm}\label{thm:dim2nef}
Consider Set-up~\ref{setup:ciSurfacenef} and fix a positive integer
\[ r \leq \frac{2}{3^{4} (3e+2) ((e+1)!)^{e}} \cdot a_{1} \cdots a_{e}. \]
For any collection $p_{1}, \ldots, p_{r} \in Y$ of $r$ points, if we write $\widetilde{Y} \colonequals \Bl_{\{ p_{1}, \ldots, p_{r} \}}Y \xrightarrow{\mu} Y$ for the blow-up with exceptional divisor $E_{i}$ over $p_{i}$, then the divisor
\[ \mu^{\ast}\O_{Y}(a_{e}) - \sum_{i=1}^{r} 3E_{i}\qquad \text{is nef on } \widetilde{Y}. \]
\end{thm}
\end{setup}

\noindent Granting Theorem~\ref{thm:dim2nef} for now, we will use it to prove Theorem~\ref{thm:cgCIsurfaces}.

\begin{proof}[Proof of Theorem~\ref{thm:cgCIsurfaces}]
In the setting of Set-up~\ref{setup:ciSurfacenef}, fix $e \geq 3$, choose $d \geq (e+1)! \cdot 4(e-1) \log 4(e-1)$, and set
\[ r = \left\lfloor \frac{2}{3^{4} (3e+2) ((e+1)!)^{e}} \cdot a_{1} \cdots a_{e} \right\rfloor. \]
Theorem~\ref{thm:dim2nef} shows that for any tuple of $r$ distinct points $p_{1}, \ldots, p_{r} \in X$, the divisor
\[ L \colonequals a_{e}H - \sum_{i=1}^{r} (n+1)E_{i} \]
on the blow-up $\mu \colon \widetilde{Y} \rightarrow Y$ is nef. It is straightforward to check that $(L^{3}) > 0$ on $\widetilde{Y}$, so $L$ is also big. By Proposition~\ref{prop:cg+nef},
\[ \cg(X) \geq r+1 \geq \frac{2}{3^{4} (3e+2) ((e+1)!)^{e}} \cdot a_{1} \cdots a_{e}. \]

So far, this only gives lower bounds on the covering gonality of complete intersection surfaces of special degrees. Now consider a very general complete intersection surface $X' \subset \P^{e+2}$ of type $(d_{1}, \ldots, d_{e})$ where $d_{1} \leq \cdots \leq d_{e}$. If $d_{1}$ is sufficiently large, then we may choose $a_{1}, \ldots, a_{e}$ such that
\[ (a_{1}, \ldots, a_{e-1}, a_{e}) \colonequals ((e+1)! q_{1}, \ldots, (e+1)! q_{e-1}, d_{e}), \]
the inegers $q_{i}$ are distinct primes, and $d_{i}/2 < a_{i} \leq d_{i}$ for $1 \leq i \leq e-1$.\footnote{In Remark~\ref{remark:primes}, we will explain how to choose these $a_{i}$.} Next, degenerate $X'$ to a union of varieties, with one component consisting of a very general complete intersection $X \subset \P^{e+2}$ of type $(a_{1}, \ldots, a_{e})$. The family can be chosen so that the total space is irreducible. In order to prove that
\[ \cg(X') \geq \cg(X), \]
we need a strengthened version of \cite[Proposition 2.2]{GK19} involving families where the central fiber is possibly reducible:

\begin{proposition}\label{propCgfamilies}
Let $f \colon \X \rightarrow T$ be a flat family of irreducible projective varieties over an irreducible one-dimensional base. Assume that the total space $\X$ is irreducible and suppose that for all $0 \not= t \in T$, the fiber $\X_{t}$ has covering gonality $\leq d$. Then every component $\X' \subset \X_{0}^{\text{red}}$ of the reduced special fiber has covering gonality $\leq d$.
\end{proposition}

\begin{proof}[Proof of Proposition~\ref{propCgfamilies}]
Adaptating the proof of \cite[Proposition 2.2]{GK19}, the key point is that irreducibility of the total space of the family means that if the base of the covering family coming from the compactified Kontsevich moduli space of stable maps is irreducible (this may be assumed), then the covering family \textit{automatically} covers every component of the central fiber. The rest of argument follows through.
\end{proof}

By the proposition above, we have
\[ \cg(X') \geq \cg(X) \geq \frac{2}{3^{4} (3e+2)((e+1)!)^{e}} \cdot \frac{d_{1}}{2} \cdots \frac{d_{e-1}}{2} \cdot d_{e}, \]
which simplifies to give the desired bound with a constant of
\[ B(e) = \frac{2}{3^{4} (3e+2)((e+1)!)^{e}2^{e-1}} \]
This completes the proof of Theorem~\ref{thm:cgCIsurfaces}.
\end{proof}

\begin{remark}
Note that Theorems \ref{thm:codim2nef} and \ref{thm:dim2nef} are false without the very general assumption. For instance, one may take a complete intersection $Y$ which contains a line and then choose points which all lie on the line.
\end{remark}

We will now give a conceptual outline of the proofs of Theorems~\ref{thm:codim2nef} and \ref{thm:dim2nef}. Proceeding by contradiction, the failure of the line bundle $L$ on $\widetilde{Y}$ to be nef means that there exists a curve $\widetilde{C} \subset \widetilde{Y}$ which intersects negatively against $L$. By projecting to the complete intersection $Y$, this roughly says that the image curve $C \colonequals \mu(\widetilde{C})$ passes through the points $p_{i}$ with large multiplicities. We then relate this to the geometry of curves on very general complete intersections to reach a contradiction. For Theorem~\ref{thm:codim2nef}, we will need lower bounds on the geometric genus of $C$, which follow from work of Ein \cite{Ein88} and Voisin \cite{Voisin96}. The argument for Theorem~\ref{thm:dim2nef} will require (i) a more precise estimate of the arithmetic genus of $C$ and its relationship to the multiplicities $m_{i}$, and (ii) lower bounds on the degree of $C$. The first ingredient (i) will incorporate some ideas originally due to Castelnuovo \cite{Castelnuovo1893} about estimating the genus of a space curve, while (ii) will involve certain degeneration arguments of Koll\'{a}r \cite{Kollar90}.

%
%

\section{Numerical invariants of curves on complete intersections}

In this section, we will collect some results about the geometry of curves in complete intersections, which will be used in the proofs of Theorems~\ref{thm:codim2nef} and \ref{thm:dim2nef}. We begin by giving lower bounds for the geometric genus of curves on generic complete intersections, which arise from calculations of Ein \cite{Ein88} and Voisin \cite{Voisin96} (for comparison, see \cite[proof of Proposition 3.8]{BDELU17}):

\begin{proposition}\label{genusHyp}
Let $X \subset \P^{n+e}$ be a very general complete intersection of dimension $n \geq 2$ and type $(d_{1}, \ldots, d_{e})$. For any integral curve $C \subset X$, we have
\[ p_{g}(C) \geq 1 + \frac{1}{2} (\sum_{i=1}^{e} d_{i} - 2n - e) \cdot \deg_{\P^{n+e}}(C). \]
\end{proposition}

\begin{proof}

Consider the spaces $V^{d_{i}} \colonequals H^{0}(\P^{n}, \O_{\P^{n+k}}(d_{i}))$ for $d_{i} \geq 2$ and let $V = \prod_{i} V^{d_{i}}$. Consider the universal complete intersection $\X \subseteq V \times \P^{n+k}$ of type $(d_{1}, \ldots, d_{e})$ with the two projections $pr_{1} \colon \X \longrightarrow V$ and $pr_{2} \colon \X \longrightarrow \P^{n+e}$. Let $v = \dim V$ and suppose that a very general complete intersection of type $(d_{1}, \ldots, d_{e})$ in $\P^{n+e}$ contains an irreducible curve of geometric genus $g$. By standard arguments, there is a diagram
\begin{center}
\begin{tikzcd}
\cC \arrow[d, swap, "\pi"] \arrow[r, "f"] & \X \arrow[d, "pr_{1}"] \\
T \arrow[r, "\rho"] & V
\end{tikzcd}
\end{center}
where $\pi \colon \cC \rightarrow T$ is a family of curves of geometric genus $g$ whose general member $\cC_{t} = \pi^{-1}(t)$ is smooth, $\rho$ is \'{e}tale, and $f_{t} \colon \cC_{t} \rightarrow X_{\rho(t)}$ is birational onto its image. In this setting, Ein and Voisin show that if $t \in T$ is a general point, then
\[ \Omega_{\cC}^{v+1} \otimes \left( (pr_{2} \circ f)^{\ast}\O_{\P^{n+e}}(2(n+e) - \sum_{i=1}^{e} d_{i} - e) \middle) \right|_{\cC_{t}} \]
is generically generated by its global sections. This implies that the canonical bundle of the general curve $\cC_{t}$ is of the form
\[ K_{\cC_{t}} \cong (\sum_{i=1}^{e} d_{i} - 2n - e) H_{\cC_{t}} + (\text{Effective}), \]
where $H_{\cC_{t}}$ is the pull-back of the hyperplane bundle from $\P^{n+e}$. Comparing degrees on both sides, we arrive at the desired result. \qedhere

\end{proof}

We will also need the following result:

\begin{lemma}\label{eqMult}
Let $C \subset \P^{N}$ (for $N \geq 3$) be a reduced and irreducible curve of degree $k$ with a finite collection of points $p_{i}$ ($i = 1, \ldots, \ell$) which have multiplicity $m_{i}$. After a generic projection of $\varphi \colon C \rightarrow C' \subset \P^{2}$, the multiplicities of the image points $\varphi(p_{i})$ in $C'$ remain the same. This leads to the estimate:
\begin{equation*}
\sum_{i=1}^{\ell} \frac{m_{i}(m_{i}-1)}{2} \leq p_{a}(C') - p_{g}(C') = \frac{(k-1)(k-2)}{2} - p_{g}(C).
\end{equation*}
\end{lemma}

\begin{remark}
Given a smooth variety $X$ and a curve $C \subset X$, the multiplicity of $C$ at a point $p$ is equal to the intersection of the strict transform $\widetilde{C}$ against the exceptional divisor $E_{p}$ of the blow-up $\mu \colon \widetilde{X} \rightarrow X$ at $p$ (see \cite[pg. 79]{Fulton98}).
\end{remark}

In order to prove Theorem~\ref{thm:dim2nef}, we will need a finer analysis of the contribution of the multiplicity of a singular point to the arithmetic genus of a curve. This is captured in:

\begin{proposition}\label{prop:multiplicityDeltaInvariant}
Let $V$ be a smooth variety of dimension $n$. Let $C \subset V$ be an irreducible and reduced curve with a singular point $p$ of multiplicity $m \colonequals \mult_{p}C$. Then the discrepancy between the arithmetic genus and the geometric genus of $C$ is bounded from below by
\[ p_{a}(C) - p_{g}(C) \geq \frac{(n-1)^{\frac{n}{n-1}}}{n} \cdot m^{\frac{n}{n-1}} - nm. \]
\end{proposition}

\begin{proof}

Consider the blow-up $\widetilde{V} \colonequals \Bl_{p} V \xrightarrow{\pi} V$, with exceptional divisor $E$. Let $\widetilde{C} \subset \widetilde{V}$ be the strict transform of $C$. We may pushforward the ideal sheaf sequence for $\widetilde{C} \subset \widetilde{V}$ along $\pi$ to obtain
\[ 0 \rightarrow \I_{C} \rightarrow \O_{V} \rightarrow \pi_{\ast}\O_{\widetilde{C}} \rightarrow \tau \colonequals R^{1} \pi_{\ast} \I_{\widetilde{C}} \rightarrow 0. \]
To compute the length of the sheaf $\tau$, we will use the theorem on formal functions. Write $\E_{\ell}$ for the $\ell$-th infinitesimal neighborhood of $E$. Then
\[ \left(R^{1}\mu_{\ast}\I_{\widetilde{C}} \right)^{\wedge} = \varprojlim_{\ell} H^{1}\left(\I_{\widetilde{C}} \otimes \O_{\E_{\ell}} \right), \]
so
\[ \text{length}(\tau) = \lim_{\ell \rightarrow \infty} h^{1}\left(\I_{\widetilde{C}} \otimes \O_{\E_{\ell}}\right). \]

The spaces on the right can be studied using the sequences:
\[ 0 \rightarrow \I_{\widetilde{C}} \big|_{E}(\ell-1) \rightarrow \I_{\widetilde{C}} \otimes \O_{\E_{\ell}} \rightarrow \I_{\widetilde{C}} \otimes \O_{\E_{\ell-1}} \rightarrow 0. \]
Define
\begin{align*}
k_{\ell} &= \dim \ker \left( H^{0}\left(\I_{\widetilde{C}} \otimes \O_{\E_{\ell-1}} \right) \rightarrow H^{1}\left(\I_{\widetilde{C}} \big|_{E}(\ell-1) \right) \right) \\
c_{\ell} &= \dim \coker \left( H^{0}(\I_{\widetilde{C}} \otimes \O_{\E_{\ell-1}}) \rightarrow H^{1}\left(\I_{\widetilde{C}} \big|_{E}(\ell-1) \right) \right).
\end{align*}
Since all higher cohomology $H^{i}$ terms are zero for $i \geq 2$, it follows that
\begin{align}\label{eq:relations}
\begin{split}
h^{0}\left(\I_{\widetilde{C}} \otimes \O_{\E_{\ell}}\right) &= h^{0}\left(\I_{\widetilde{C}} \big|_{E}(\ell-1) \right) + k_{\ell}, \\
h^{1}\left(\I_{\widetilde{C}} \otimes \O_{\E_{\ell}} \right) &= c_{\ell} + h^{1}\left(\I_{\widetilde{C}} \otimes \O_{\E_{\ell-1}} \right).
\end{split}
\end{align}
In addition,
\begin{align}\label{eq:relationcandk}
c_{\ell} - k_{\ell} = h^{1}\left(\I_{\widetilde{C}} \big|_{E}(\ell-1) \right) - h^{0}\left(\I_{\widetilde{C}} \otimes \O_{\E_{\ell-1}} \right).
\end{align}

Now we will need the following:
\begin{claim}
With the set-up above:
\begin{enumerate}[(i)]
\itemsep0em 
\item The dimensions $h^{0}\left(\I_{\widetilde{C}} \otimes \O_{\ell-1} \right)$ are non-decreasing in $\ell$ and stabilize for $\ell \gg 0$.
\item For any $\ell > 0$,
\[ h^{1}\left(\I_{\widetilde{C}} \otimes \O_{\E_{\ell}} \right) = - \sum_{i=0}^{\ell-2} \chi(\I_{\widetilde{C}} \big|_{E}(i)) + h^{1}(\I_{\widetilde{C}} \big|_{E}(\ell-1)) + k_{\ell}. \]
\end{enumerate}
\end{claim}
\noindent Granting the Claim for now, we will use it to complete the proof of Proposition~\ref{prop:multiplicityDeltaInvariant}. Observe that
\[ \chi \left(\I_{\widetilde{C}}\big|_{E}(i) \right) = \chi \left(\O_{E}(i) \right) - \chi \left(\O_{\widetilde{C} \cap E}(i) \right) = h^{0}\left(\O_{E}(i) \right) - \mult_{p}C = \binom{i+n-1}{n-1} - m. \]
Negate both sides to get
\begin{equation}\label{eq:negChi}
-\chi\left(\I_{\widetilde{C}} \big|_{E}(i) \right) = m - \binom{i+n-1}{n-1},
\end{equation}
and let $\ell_{0}$ be the largest integer such that
\[ m - \binom{\ell_{0}+n-1}{n-1} \geq 0. \]
Then by definition
\begin{equation}\label{eq:lowerell0}
m - \binom{\ell_{0}+n}{n-1} < 0 \implies \frac{(\ell_{0}+n)^{n-1}}{n-1} > m \implies \ell_{0} > \left( (n-1)m \right)^{\frac{1}{n-1}}-n.
\end{equation}
We will now use \eqref{eq:negChi} to add up the contributions of each term:
\[ - \sum_{i=0}^{\ell-2} \chi \left(\I_{\widetilde{C}} \big|_{E}(i) \right) = \sum_{i=0}^{\ell_{0}} \left[ m - \binom{i+n-1}{n-1} \right] = m \cdot \ell_{0} - \binom{\ell_{0}+n}{n} \geq m \cdot \left( \frac{n-1}{n} \cdot \ell_{0} - 1 \right). \]
By the Claim and \eqref{eq:lowerell0}, it follows that
\[ \text{length}(\tau) > m \cdot \left( \frac{(n-1)^{\frac{n}{n-1}}}{n} \cdot m^{\frac{1}{n-1}} - n \right) = \frac{(n-1)^{\frac{n}{n-1}}}{n} \cdot m^{\frac{n}{n-1}} - nm. \]

As for the Claim, part (i) follows from the equation $h^{1}(\I_{\widetilde{C}} \otimes \O_{\E_{\ell}}) = c_{\ell} + h^{1}(\I_{\widetilde{C}} \otimes \O_{\E_{\ell-1}})$ and the fact that
\[ H^{1}\left(\I_{\widetilde{C}} \big|_{E}(\ell-1) \right) = 0 \qquad \text{for } \ell \gg 0. \]
For part (ii), we will argue by induction. The base case $\ell = 1$ follows from relations \eqref{eq:relations} and \eqref{eq:relationcandk}. Now assume that the equation holds for some positive integer $\ell = j-1$. By these same relations,
\begin{align*}
h^{1}\left(\I_{\widetilde{C}} \otimes \O_{\E_{j}} \right) &= h^{1}\left(\I_{\widetilde{C}} \otimes \O_{\E_{j-1}} \right) + c_{j} \\
&= h^{1}\left(\I_{\widetilde{C}} \otimes \O_{\E_{j-1}} \right) + h^{1}\left(\I_{\widetilde{C}} \big|_{E}(j-1) \right) - h^{0}\left(\I_{\widetilde{C}} \otimes \O_{\E_{j-1}} \right) + k_{j} \\
&= h^{1}\left(\I_{\widetilde{C}} \otimes \O_{\E_{j-1}} \right) + h^{1}\left(\I_{\widetilde{C}} \big|_{E}(j-1) \right) - \left( h^{0}\left(\I_{\widetilde{C}} \big|_{E}(j-2) \right) + k_{j-1} \right) + k_{j}
\end{align*}
We may rewrite the inductive hypothesis as
\[ h^{1}\left(\I_{\widetilde{C}} \otimes \O_{\E_{j-1}} \right) - h^{0}\left(\I_{\widetilde{C}} \big|_{E}(j-2)\right) = - \sum_{i=0}^{j-2} \chi \left(\I_{\widetilde{C}} \big|_{E}(i) \right) + k_{j-1}. \]
Therefore,
\[ h^{1}\left(\I_{\widetilde{C}} \otimes \O_{\E_{j}} \right) = - \sum_{i=0}^{j-2} \chi\left(\I_{\widetilde{C}} \big|_{E}(i) \right) + h^{1}\left(\I_{\widetilde{C}} \big|_{E}(j-1) \right) + k_{j}, \]
which is what we want.
\end{proof}

For dimension counts, the following expression will be useful:
\begin{lemma}\label{lem:globalsectionsKoszul}
Let $Y \subset \P^{n+f}$ be a complete intersection of dimension $n \geq 2$ and type $(a_{1}, \ldots, a_{f})$. Then
\[ h^{0}(Y, \O(\ell)) = \sum_{j_{1}=0}^{a_{1}-1} \cdots \sum_{j_{f}=0}^{a_{f}-1} h^{0}\left( \P^{n}, \O(\ell - j_{1} - \cdots - j_{f}) \right). \]
\end{lemma}

\begin{proof}
Consider the Koszul resolution for a complete intersection $Y \subset \P^{n+f}$ of dimension $n \geq 2$ and type $(a_{1}, \ldots, a_{f})$. Using the vanishing of higher cohomology of line bundles on $\P^{n+f}$, we have
\begin{align*}
h^{0}(Y, \O(\ell)) &= h^{0}(\P^{n+f}, \O(\ell)) - \sum_{j=1}^{f} h^{0}(\P^{n+f}, \O(\ell - a_{j})) \\
&\quad + \sum_{1 \leq u < v \leq n+1} h^{0}(\P^{n+f}, \O(\ell-a_{u} - a_{v})) - \cdots
\end{align*}
We may rearrange the terms in pairs as follows:
\begin{align*}
&[h^{0}(\P^{n+f}, \O(\ell)) - h^{0}(\P^{n+f}, \O(\ell-a_{1}))] \\
&\qquad - \sum_{j=2}^{f} [h^{0}(\P^{n+f}, \O(\ell-a_{j})) - h^{0}(\P^{n+f}, \O(\ell - a_{j} - a_{1}))] + \cdots \\
&= \left[ \binom{\ell+n+f}{n+f} - \binom{\ell - a_{1}+n+f}{n+f} \right] \\
&\qquad - \sum_{j=2}^{f} \left[ \binom{\ell - a_{j} + n + f}{n+f} - \binom{\ell-a_{j} - a_{1}+n+f}{n+f} \right] + \cdots,
\end{align*}
where the convention we adopt is that the binomial coefficients are zero if the upper index is smaller than the lower index. The expression in each bracket can be replaced by a sum of binomial coefficients. For instance,
\begin{align*}
\binom{\ell+n+f}{n+f} - \binom{\ell - a_{1}+n+f}{n+f} &= \sum_{j=1}^{a_{1}} \binom{j+\ell - a_{1} + n + f - 1}{n+f-1} \\
&= \sum_{j=0}^{a_{1}-1} \binom{\ell - j + n + f - 1}{n+f-1}.
\end{align*}
Similarly, the next term can be rewritten as
\[ \binom{\ell - a_{j} + n + f}{n+f} - \binom{\ell-a_{j} - a_{1}+n+f}{n+f} = \sum_{j=0}^{a_{1}-1} \binom{\ell - j - a_{j} + n + f - 1}{n+f-1}, \]
and so on. Adding up the contribution of each term, it follows that
\[ h^{0}(Y, \O(\ell)) = \sum_{j=0}^{a_{1}-1} h^{0}(W, \O(\ell-j)), \]
where $W \subset \P^{n+f-1}$ is a complete intersection of dimension $n$ and type $(a_{2}, \ldots, a_{f})$ (compare with the Koszul resolution for $W$). By repeating this process, we obtain
\[ h^{0}(Y, \O(\ell)) = \sum_{j_{1}=0}^{a_{1}-1} \cdots \sum_{j_{f}=0}^{a_{f}-1} h^{0}(\P^{n}, \O(\ell - j_{1} - \cdots - j_{f})). \qedhere \]
\end{proof}

Another ingredient that goes into the proof of Theorem~\ref{thm:dim2nef} is an estimate of the minimal degree of curves contained in a generic complete intersection of special type (see Proposition~\ref{degreeCurves}). The following results are a straightforward generalization of those in \cite{Kollar90}. One may also compare with \cite[\S 2]{Paulsen21}:

\begin{lemma}\label{divisibleCurves}
Let $d, k$ be integers. Assume that there is a smooth projective variety $Y$ of dimension $N$ with a line bundle $L$ such that $L^{N} = d$ and $k \mid B \cdot L$ holds for every curves $B \subset Y$. If $X \subset \P^{N+1}$ is a very general hypersurface of degree $d$ and $C \subset X$ is any curve, then $k \big| N! \cdot \deg C$.
\end{lemma}

\begin{proof}

We embed $Y$ by $L$ to get $Y \subset \P^{N}$ and project generically to get a finite morphism
\[ \varphi \colon Y \rightarrow \widetilde{Y} \subset \P^{n+1}. \]
We know that $\varphi$ is an isomorphism on an open set, ramified $2:1$ over a divisor in $\widetilde{Y}$, and so on up to $n:1$ over a curve. So for any irreducible curve $C \subset \widetilde{Y}$, the projection formula gives
\[ \varphi^{-1}(C) \cdot L = \deg(\varphi^{-1}(C) \rightarrow C) \cdot \deg_{\P^{n+1}} C \]
where $\deg(\varphi^{-1}(C) \rightarrow C)$ is some positive integer less than or equal to $n$. By our hypothesis, it follows that
\[ k \mid \varphi^{-1}(C) \cdot L \mid n! \cdot \deg C. \]
If $X \subset \P^{n+1}$ is a very general hypersurface of degree $d$ and $C_{X} \subset X$ is a curve, then $C_{X}$ can be specialized to some curve $C_{\widetilde{Y}} \subset \widetilde{Y}$. Therefore, $k \mid n! \cdot \deg C_{X}$.
\end{proof}

\begin{example}[Van Geemen]\label{ex:AV}
Let $(A, L)$ be a very general Abelian variety of dimension $m \geq 3$ with a polarization of type $(1, \ldots, 1, \delta)$. Then $\NS(A) = \Z[L]$ and $L^{m} = m! \cdot \delta$. In $H^{1}(A, \Z)$, choose a basis $\{ dx_{i} \}$ such that
\[ c_{1}(L) = dx_{1} \wedge dx_{m+1} + \cdots + dx_{m-1} \wedge dx_{2m-1} + \delta \cdot dx_{m} \wedge dx_{2m}. \]
Write $\omega_{i} \colonequals dx_{i} \wedge dx_{m+i}$ for $i = 1, \ldots, m$. Then
\[ \frac{c_{1}(L)^{\wedge (m-1)}}{(m-1)!} = \omega_{1} \wedge \omega_{2} \wedge \cdots \wedge \omega_{m-1} + \delta \cdot \sum_{i=1}^{m-1} \omega_{1} \wedge \cdots \wedge \widehat{\omega_{i}} \wedge \cdots \wedge \omega_{m} \]
represents an indivisible integral class in $H^{2m-2}(A, \Z)$. By the Hard Lefschetz theorem, the cohomology class of every curve in $A$ is a rational multiple of $c_{1}(L)^{\wedge (m-1)}$ and hence an integral multiple of $c_{1}(L)^{\wedge (m-1)}/(m-1)!$, so $m \delta \mid C \cdot L$ for every curve $C \subset A$.
\end{example}

\begin{proposition}\label{degreeCurves}
Fix integers $q_{1}, \ldots, q_{f} > 2^{n+f-1}$ such that $(q_{i}, q_{j}) = 1$ for all $1 \leq i < j \leq f$. Let $X \subset \P^{n+f}$ be a very general complete intersection of dimension $n$ and type
\[ ((n+f-1)! \cdot q_{1}, \ldots, (n+f-1)! \cdot q_{f}). \]
If $C \subset X$ is a curve, then $q_{1} \cdots q_{f} \mid (n+f-2)! \cdot \deg_{\P^{n+f}} C$.
\end{proposition}

\begin{proof}

Fix $i$ and let $q_{i} > 2^{n+f-1}$ be an integer as above. If $(A, L)$ is a very general abelian variety of dimension $n+f-1$ with a polarization of type $(1, \ldots, 1, q_{i})$, then Debarre-Hulek-Spandaw \cite{DHS94} have shown that $L$ is very ample. Example~\ref{ex:AV} demonstrates that $(n+f-1) \cdot q_{i} \mid C \cdot L$ for every curve $C \subset A$. By applying the pair $(A, L)$ to Lemma~\ref{divisibleCurves} and using the fact that $L^{n+f-1} = (n+f-1)! \cdot q_{i}$, we see that if $X_{(n+f-1)! \cdot q_{i}} \subset \P^{n+f}$ is a generic hypersurface of degree $(n+f-1)! \cdot q_{i}$ and $C \subset X_{(n+f-1)! \cdot q_{i}}$ is any curve, then
\[ (n+f-1)q_{i} \mid (n+f-1)! \cdot \deg C \implies q_{i} \mid (n+f-2)! \cdot \deg C. \]
Now we can vary this argument for all $i = 1, \ldots, f$; in other words, we will apply it to the $f$ hypersurfaces that intersect to give $X$. From the hypothesis that $(q_{i}, q_{j}) = 1$ for all $1 \leq i < j \leq f$, it follows that for any curve $C \subset X$,
\[ q_{1} \cdots q_{f} \mid (n+f-2)! \cdot \deg C. \qedhere \]
\end{proof}

Proposition~\ref{degreeCurves} suggests that there should be uniform lower bounds for the degrees of curves in very general complete intersections of \text{any} type $(d_{1}, \ldots, d_{k})$, i.e. the following should hold:

\begin{conj}
Given a very general complete intersection variety $X \subset \P^{n+f}$ of dimension $n$ and type $(d_{1}, d_{2}, \ldots, d_{f})$, any curve in $X$ has degree $\geq D \cdot d_{1} \cdots d_{f}$, where $D$ is some positive constant depending only on $n$ and $f$.
\end{conj}

\noindent Note that a positive answer to the question above would streamline the proof of Theorem~\ref{thm:cgCIsurfaces} by eliminating the need for Proposition~\ref{propCgfamilies} and Remark~\ref{remark:primes}.

%
%

\section{Proof of Theorem~\ref{thm:codim2nef}}

Recalling Set-up~\ref{setup:codim2nef}, we will induct on the number of points $r$. For the base case, the statement is trivial for $r \leq 2$ as soon as $b \geq 2$ since $\O_{Y}(1)$ is very ample. By induction, we may assume that the theorem holds for $r = s$ where
\[ 2 \leq s \leq \frac{2}{3(n+1)^{2}} ab - 1. \]
We want to prove that the theorem holds for $r = s+1$.

Suppose for the sake of contradiction that the theorem fails when $r = s+1$. Then $\exists p_{1}, \ldots, p_{s+1} \in Y$ such that the corresponding divisor
\[ L := bH - \sum_{i=1}^{s+1} (n+1) E_{i} \]
on the blow-up $\widetilde{Y} \xrightarrow{\mu} Y$ is not nef. Here, $H = \mu^{\ast}\O_{Y}(1)$ and $E_{i}$ is the exceptional divisor over $p_{i}$.

By definition, this means that there is an integral curve $\widetilde{C} \subset \widetilde{Y}$ such that
\[ L \cdot \widetilde{C} < 0. \]
We claim that $\widetilde{C}$ cannot be contained in some exceptional divisor $E_{j}$, because otherwise $\widetilde{C} \cdot E_{j} < 0$ would imply $\widetilde{C} \cdot L > 0$, which is a contradiction. By the Lefschetz hyperplane theorem and Poincar\'{e} duality, $\widetilde{C}$ is numerically a $\Q$-linear combination of terms involving $H^{n}$ and $E_{i}^{n}$ (note that the mixed terms involving $H \cdot E_{i}$ must vanish because we have blown up a collection of points). One can check that $H^{n+1} = a$ and $(-E_{i})^{n+1} = -1$. Furthermore, the intersection numbers $\widetilde{C} \cdot H \geq 1$ and $\widetilde{C} \cdot E_{i} \geq 0$ must be integers. It follows that the numerical class of $\widetilde{C}$ is given by
\begin{equation}\label{eq:numClassCcodim2}
\widetilde{C} \equiv_{\text{num}} \frac{k}{a} H^{n} + (-1)^{n} \sum_{i=1}^{s+1} m_{i} E_{i}^{n}
\end{equation}
where $k \geq 1$ is the degree of the image curve $C \colonequals \mu(\widetilde{C}) \subset \P^{n+2}$ and $m_{i} \geq 0$ are the multiplicities of $C$ at $p_{i}$. Note that
\[ L \cdot \widetilde{C} < 0 \implies \sum_{i=1}^{s+1} m_{i} > \frac{1}{n+1} bk. \]
Since the quantity $\sum m_{i}^{2}$ is minimized when all of the $m_{i}$ are the same, it follows that
\begin{equation}\label{eq:lowerBound}
\sum_{i=1}^{s+1} m_{i} > \frac{1}{n+1}kb =: \gamma \implies \sum_{i=1}^{s+1} m_{i}^{2} > \left( \frac{\gamma}{s+1} \right)^{2} \cdot (s+1) = \frac{k^{2} b^{2}}{(n+1)^{2} (s+1)}.
\end{equation}

On the other hand, our induction hypothesis implies that the divisor $L_{I} := bH_{Y} - \sum_{i \in I} (n+1)E_{i}$ is nef for any subset $I \subset \{ 1, 2, \ldots, s+1 \}$ with $\# I = s$. Averaging over all $I$ shows that
\[ L_{s+1} := \frac{s+1}{s} bH_{Y} - \sum_{i=1}^{s+1} (n+1)E_{i} \]
is nef. This implies that $L_{s+1} \cdot \widetilde{C} \geq 0$, and hence
\begin{equation}\label{sumMult2}
\sum_{i=1}^{s+1} m_{i} \leq \frac{(s+1)}{s (n+1)} bk.
\end{equation}

By Lemma~\ref{eqMult} and Proposition~\ref{genusHyp} applied to $C \subset Y \subset \P^{n+2}$, we have
\begin{align}\label{boundPlane+Voisin}
\sum_{i=1}^{r} \frac{m_{i}(m_{i}-1)}{2} &\leq \frac{1}{2} (k-1)(k-2) - p_{g}(C) \nonumber \\
&\leq \frac{1}{2} (k-1)(k-2) - \frac{1}{2} (a-2n-3) k - 1 \nonumber \\
&= \frac{1}{2} (k^{2} +(2n-a)k)
\end{align}
Next, we can combine this with the inequalities in \eqref{eq:lowerBound} and \eqref{sumMult2}:
\[ \frac{k^{2} b^{2}}{(n+1)^{2} (s+1)} < \sum_{i=1}^{r} m_{i}^{2} = \sum_{i=1}^{r} m_{i}(m_{i}-1) + \sum_{i=1}^{r} m_{i} \leq k^{2} +(2n-a)k + \frac{(s+1)}{s(n+1)} bk. \]
After simplifying, we get
\begin{equation}\label{eq:solvingk}
\left[ \frac{b^{2}}{(n+1)^{2}\cdot (s+1)} - 1 \right] \cdot k < 2n - a + \frac{(s+1)}{s (n+1)}b.
\end{equation}
Recall from our induction set-up that
\[ 3 \leq s+1 \leq \frac{2}{3(n+1)^{2}} ab \quad \text{and} \quad n \geq 2 \implies \frac{(s+1)}{s (n+1)} \leq \frac{1}{3}. \]
In addition, $b \geq a$ so term on the left hand side of \eqref{eq:solvingk} is positive. Solving for $k$ yields:
\begin{equation}\label{UpperBound-k}
k < \frac{(2n-a+b/3)\cdot 2a}{3b-2a} \leq \frac{1}{9} \left( 2a + \frac{(18n-7a)\cdot 2a}{3b-2a} \right) \leq \frac{2}{9} a.
\end{equation}
Since all $m_{i}$ are nonnegative integers (in fact, the induction hypothesis tells us that $m_{i} \geq 1$ for all $i$), the inequality in \eqref{boundPlane+Voisin} also gives
\[ \frac{1}{2} (k-1)(k-2) - \frac{1}{2}k (a-2n-3) - 1 \geq \sum_{i=1}^{r} \frac{m_{i}(m_{i}-1)}{2} \geq 0 \implies k \geq a-2n, \]
which contradicts \eqref{UpperBound-k} as soon as $a \geq 18n/7$. This completes the proof of Theorem~\ref{thm:codim2nef}.

%
%

\section{Proof of Theorem~\ref{thm:dim2nef}}

Recalling Set-up~\ref{setup:ciSurfacenef}, we will induct on the number of points $r$. Let $\alpha \colonequals a_{1} \cdots a_{e-1}$ be the degree of the threefold $Y \subset \P^{e+2}$. For the base case $r \leq 2$, the statement is trivial as soon as $d\geq 2$ since $H_{Y}$ is very ample. By induction, we may assume that the theorem holds for $r = s$ where
\[ 2 \leq s \leq \frac{2}{3^{4} (3e+2) ((e+1)!)^{e}} \cdot \alpha a_{e} -1. \]
We want to prove that the theorem holds for $r = s+1$.

Suppose for the sake of contradiction that the theorem fails when $r = s+1$. Then $\exists p_{1}, \ldots, p_{s+1} \in Y$ such that the corresponding divisor
\[ L := a_{e}H - \sum_{i=1}^{s+1} 3 E_{i} \]
on the blow-up $\widetilde{Y}$ is not nef.

By definition, this means that there is an integral curve $\widetilde{C} \subset \widetilde{Y}$ such that
\[ L \cdot \widetilde{C} < 0. \]
As before, note that $\widetilde{C}$ cannot be contained in some exceptional divisor $E_{j}$. By the Lefschetz hyperplane theorem and Poincar\'{e} duality, $\widetilde{C}$ is numerically a $\Q$-linear combination of terms involving $H^{2}$ and $E_{i}^{2}$ (note that the mixed terms involving $H \cdot E_{i}$ must vanish because we have blown up a collection of points). One can check that $H^{3} = \alpha$ and $(-E_{i})^{3} = -1$. Furthermore, the intersection numbers $\widetilde{C} \cdot H \geq 1$ and $\widetilde{C} \cdot E_{i} \geq 0$ must be integers. It follows that the numerical class of $\widetilde{C}$ is given by
\begin{equation}\label{eq:numClassCsurface}
\widetilde{C} \equiv_{\text{num}} \frac{k}{\alpha} H^{2} + \sum_{i=1}^{s+1} m_{i} E_{i}^{2}
\end{equation}
for some integers $k \geq 1$ and $m_{i} \geq 0$.

\begin{remark}\label{remark:kAndm_i}
Note that $k$ is the degree of the image curve $C \colonequals \mu(\widetilde{C}) \subset Y \subset \P^{e+2}$ and the integers $m_{i}$ are the multiplicities of $C$ at $p_{i}$. By the induction hypothesis, we know that $m_{i} \geq 1$ for all $i$.
\end{remark}

Using the description \eqref{eq:numClassCsurface} for $\widetilde{C}$, the condition $L \cdot \widetilde{C} < 0$ reduces to
\begin{equation}\label{equation:CintersectedL<0}
\sum_{i=1}^{s+1} m_{i} > \frac{1}{3} a_{e}k.
\end{equation}
For a fixed $\sum_{i=1}^{s+1} m_{i}$, the quantity $\sum m_{i}^{3/2}$ is minimized when all of the $m_{i}$ are the same, so setting $N \colonequals a_{e}k/3$ for the right hand side yields
\begin{equation}\label{lowerBoundm3/2}
\sum_{i=1}^{s+1} m_{i}^{3/2} > \left( \frac{N}{s+1} \right)^{3/2} \cdot (s+1) = \frac{1}{3^{3/2} \cdot (s+1)^{1/2}} k^{3/2} a_{e}^{3/2}.
\end{equation}

The rest of the proof will be devoted to bounding the expression $\sum_{i=1}^{s+1} m_{i}^{3/2}$ from above. We will first use a dimension count to approximate $C$ as a complete intersection curve. The key point is that if we can find an effective divisor $V \in \abs{\O_{Y}(\ell)}$ of degree $\ell \leq \frac{1}{3} a_{e}$ passing through all of the points $p_{i}$, then
\[ L \cdot \widetilde{C} < 0 \implies \widetilde{V} \cdot \widetilde{C} < 0, \]
where $\widetilde{V}$ is the strict transform of $V$. It must then follow that $\widetilde{C} \subset \widetilde{V}$. This is the basic idea behind:

\begin{lemma}\label{lemma: IntersectV1V2}
With Set-up~\ref{setup:ciSurfacenef} in mind and integers $a_{e} \geq a_{e-1} \geq \cdots \geq a_{1}$, there are surfaces $V_{j} \in \abs{\O_{Y}(b_{j})}$ with $b_{j} \leq \frac{1}{3}a_{e}$ for $j = 1,2$ such that their intersection $V_{1} \cap V_{2}$ is a curve (not necessarily irreducible) which contains $C$ as an irreducible component.
\end{lemma}

\begin{proof}
Recall that
\[ s+1 \leq \frac{2}{3^{4} (3e+2) ((e+1)!)^{e}} \cdot \alpha a_{e}. \]
Choose $b_{1}$ to be the minimal degree such that there exists a hypersurface $V_{1} \in \abs{\O_{Y}(b_{1})}$ passing through all of the points $p_{i}$ ($1 \leq i \leq s+1$). We claim that $b_{1} \leq \lfloor a_{e}/3 \rfloor$. For $\ell = \lfloor a_{e}/3 \rfloor$, we can apply Lemma~\ref{lem:globalsectionsKoszul} with $f = e-1$, $n = 3$, and $Y \subset \P^{e+2}$ a threefold of type $(a_{1}, \ldots, a_{e-1})$ to see that
\begin{align*}
h^{0}(Y, \O(\ell)) &= \sum_{j_{1}=0}^{a_{1}} \cdots \sum_{j_{e-1}=0}^{a_{e-1}} h^{0}(\P^{3}, \O(\ell - j_{1} - \cdots - j_{e-1})) \\
&\geq \sum_{j_{1}=0}^{\lfloor a_{1}/(3e) \rfloor} \cdots \sum_{j_{e-1}=0}^{\lfloor a_{e-1}/(3e) \rfloor} h^{0}(\P^{3}, \O(\ell - j_{1} - \cdots - j_{e-1})) \qquad (\text{by truncating the sum}) \\
&> \sum_{j_{1}=0}^{\lfloor a_{1}/(3e) \rfloor} \cdots \sum_{j_{e-1}=0}^{\lfloor a_{e-1}/(3e) \rfloor} \frac{1}{3!} \left( \ell - \left\lfloor \frac{a_{1}}{3e} \right\rfloor - \cdots - \left\lfloor \frac{a_{e-1}}{3e} \right\rfloor + 1 \right)^{3} \\
&\geq \frac{1}{6} \cdot \frac{a_{1}}{3e} \cdots \frac{a_{e-1}}{3e} \cdot \left(\left\lfloor \frac{1}{3} a_{e} \right\rfloor - (e-1) \cdot \left\lfloor \frac{a_{e}}{3e} \right\rfloor + 1 \right)^{3} \qquad (\text{setting } \ell = \lfloor a_{e}/3 \rfloor) \\
&\geq \frac{1}{6 \cdot (3e)^{e+2}} \alpha a_{e}^{3} \\
&> s+1 \qquad (\text{for } e \geq 2).
\end{align*}
Hence, there exists a hypersurface $V_{1} \in \abs{\O_{Y}(\lfloor a_{e}/3 \rfloor)}$ of degree $b_{1} \leq \lfloor a_{e}/3 \rfloor$ which passes through all of the $p_{i}$. We claim that $C \subset V_{1}$. To see this, observe that the class of the strict transform $\widetilde{V}_{1} \subset \widetilde{Y}$ is given by
\[ \widetilde{V_{1}} \equiv_{\text{lin}} b_{1} H - \sum_{i=1}^{s+1} c_{i} E_{i} \]
for some integers $c_{i} \geq 1$ with $1 \leq i \leq s+1$. By comparing with \eqref{equation:CintersectedL<0}, the condition $b_{1} \leq a_{e}/3$ implies that $\widetilde{V_{1}} \cdot \widetilde{C} < 0$ and hence $\widetilde{C} \subset \widetilde{V_{1}}$. Therefore, $C \subset V_{1}$. Since $C$ is irreducible and passes through all of the $p_{i}$ by Remark~\ref{remark:kAndm_i}, we may assume that $V_{1}$ is irreducible.

By a similar argument, choose $b_{2}$ to be the smallest positive integer such that there exists a section $V_{2} \in \abs{\O_{V_{1}}(b_{2})}$ containing $p_{i}$ ($1 \leq i \leq s+1$). Since $V_{1}$ is a complete intersection, a similar dimension count using Lemma~\ref{lem:globalsectionsKoszul} for $V_{1} \subset \P^{e+2}$ leads to:
\begin{align*}
h^{0}(V_{1}, \O(\ell)) &= \sum_{j_{0}=0}^{b_{1}-1} \sum_{j_{1}=0}^{a_{1}-1} \cdots \sum_{j_{e-1}=0}^{a_{e-1}-1} h^{0}(\P^{2}, \O(\ell-j_{0}-j_{1}- \cdots - j_{e-1})) \\
&\geq \sum_{j_{1}=0}^{\lfloor a_{1}/(3e) \rfloor} \cdots \sum_{j_{e-1}=0}^{\lfloor a_{e-1}/(3e) \rfloor} h^{0}(\P^{2}, \O(\ell-j_{1}- \cdots - j_{e-1})) \quad \let\scriptstyle\textstyle \substack{ (\text{taking the first term } \\ j_{0}=0 \text{ and truncating}) } \\
&> \sum_{j_{1}=0}^{\lfloor a_{1}/(3e) \rfloor} \cdots \sum_{j_{e-1}=0}^{\lfloor a_{e-1}/(3e) \rfloor} \frac{1}{2} \left(\ell - \left\lfloor \frac{a_{1}}{3e} \right\rfloor - \cdots - \left\lfloor \frac{a_{e-1}}{3e} \right\rfloor + 1 \right)^{2} \quad (\text{since } \ell = \lfloor a_{e}/3 \rfloor ) \\
&\geq \frac{1}{2} \cdot \frac{a_{1}}{3e} \cdots \frac{a_{e-1}}{3e} \cdot \left( \lfloor \frac{1}{3} a_{e} \rfloor - (e-1) \cdot \lfloor \frac{a_{e}}{3e} \rfloor + 1 \right)^{2} \\
&\geq \frac{1}{2 \cdot (3e)^{e+1}} \alpha a_{e}^{2} \\
&> s+1.
\end{align*}
Thus, there exists a hypersurface $V_{2} \in \abs{\O_{V_{1}}(b_{2})}$ of degree $b_{2} \leq \lfloor a_{e}/3 \rfloor$ passing through all of the $p_{i}$. The restriction map $H^{0}(Y, \O(b_{2})) \rightarrow H^{0}(V_{1}, \O(b_{2}))$ is surjective (see \cite[Lemma VIII.9]{Beauville96}), so without loss of generality we may take $V_{2}$ to be a surface in $\abs{\O_{Y}(b_{2})}$ such that the intersection $V_{1} \cap V_{2}$ is a curve (not necessarily irreducible). By the same reasoning as above, the condition $b_{2} \leq a_{e}/3$ implies that $V_{2}$ contains $C$. This proves that $C \subset V_{1} \cap V_{2}$.
\end{proof}

We can now give a bound for the arithmetic genus of $C \subset Y$.

\begin{lemma}
With the same assumptions as in the previous lemma, the arithmetic genus of $C$ is bounded from above by
\[ p_{a}(C) \leq (a_{1} + \cdots + a_{e-1} + 2a_{e}/3) \cdot \deg C \]
\end{lemma}

\begin{proof}

The approach used here follows from ideas of Castelnuovo \cite{Castelnuovo1893} about bounding the genus of space curves as well as subsequent work of Harris \cite{Harris80}. In summary, let $\alpha_{\ell}$ be the dimension of the image of the restriction map
\[ H^{0}(Y, \O(\ell)) \xrightarrow{\rho_{\ell}} H^{0}(C, \O(\ell)). \]
We would like to bound the differences $\alpha_{\ell} - \alpha_{\ell-1}$ from below in order to get a lower bound for $h^{0}(C, \O(\ell))$. For $\ell \gg 0$ we can then apply Riemann-Roch to obtain a bound on $p_{a}(C)$.

Let $H$ be a generic hyperplane section in $\P^{e+2}$ and consider the intersections
\[ \Gamma \colonequals H \cap C, \qquad Y_{2} \colonequals Y \cap H. \]
There is another restriction map
\[ H^{0}(Y, \I_{\Gamma}(\ell)) \xrightarrow{\sigma_{\ell}} H^{0}(C, \I_{\Gamma}(\ell) \big|_{C}) \cong H^{0}(C, \O(\ell-1)). \]
Since $H^{0}(Y, \O(\ell-1))$ injects into $H^{0}(Y, \I_{\Gamma}(\ell))$ via the multiplication map by the defining equation of $H$ and the kernels of $\rho_{\ell}$ and $\sigma_{\ell}$ are isomorphic, i.e. equal to $H^{0}(Y, \I_{C}(\ell))$, we see that
\[ \beta_{\ell} \colonequals h^{0}(Y, \O_{\ell}) - h^{0}(Y, \I_{\Gamma}(\ell)) \leq \alpha_{\ell} - \alpha_{\ell-1}. \]
The natural restriction maps $H^{0}(Y, \O(\ell)) \rightarrow H^{0}(Y_{2}, \O(\ell))$ and $H^{0}(Y, \I_{\Gamma}(\ell)) \rightarrow H^{0}(Y_{2}, \I_{\Gamma}(\ell))$ are both surjective, so
\[ \beta_{\ell} = h^{0}(Y_{2}, \O(\ell)) - h^{0}(Y_{2}, \I_{\Gamma}(\ell)). \]
In order to understand the $\beta_{\ell}$, set $\gamma_{0} \colonequals \beta_{0}$ and define the differences
\[ \gamma_{\ell} \colonequals \beta_{\ell} - \beta_{\ell-1}. \]
Furthermore, take a generic hyperplane section $Y_{1} \colonequals H' \cap Y_{2}$ (note that $\dim Y_{i} = i$). Assuming $H'$ is disjoint from $\Gamma$, we can write down a long sequence on cohomology:
\begin{equation}\label{LES:gammatoe}
0 \rightarrow H^{0}(Y_{2}, \I_{\Gamma}(\ell-1)) \rightarrow H^{0}(Y_{2}, \I_{\Gamma}(\ell)) \xrightarrow{(\ast)} H^{0}(Y_{1}, \O(\ell)) \rightarrow H^{1}(Y_{2}, \I_{\Gamma}(\ell-1)).
\end{equation}
If we let $e_{\ell}$ be the dimension of the image of $(\ast)$, then one can show that
\begin{equation}\label{equation:gammatoe}
\gamma_{\ell} = h^{0}(Y_{1}, \O(\ell)) - e_{\ell}.
\end{equation}

By Serre vanishing, $\gamma_{\ell} = 0$ for $\ell \gg 0$. For our purposes, we will need an effective bound for when the numbers $\gamma_{\ell}$ become zero. This is where Lemma~\ref{lemma: IntersectV1V2} will be utilized. Recall that $C \subset V_{1} \cap V_{2}$ where $V_{j} \in \abs{\O_{Y}(b_{j})}$ are surfaces in $Y$ and $V_{1} \cap V_{2}$ is a (possibly reducible) curve. So
\[ \Gamma \subset H \cap V_{1} \cap V_{2} \]
where the right hand side is also finite. From the Koszul resolution for $H \cap V_{1} \cap V_{2} \subset Y_{2}$, Serre duality, and \cite[Lemma VIII.9]{Beauville96}, it follows that
\[ H^{1}(Y_{2}, I_{H \cap V_{1} \cap V_{2}}(\ell)) = 0 \quad \text{for} \quad \ell > a_{1} + \cdots + a_{e-1} -e-2+b_{1}+b_{2}. \]
The exact sequence $0 \rightarrow I_{H \cap V_{1} \cap V_{2}} \rightarrow I_{\Gamma} \rightarrow \O_{H \cap V_{1} \cap V_{2} \setminus \Gamma} \rightarrow 0$ on $Y_{2}$ implies that
\[ H^{1}(Y_{2}, \I_{\Gamma}(\ell)) = 0 \quad \text{if} \quad \ell > a_{1} + \cdots + a_{e-1} - e - 2 + b_{1} + b_{2}. \]
Since $b_{1}, b_{2} \leq a_{e}/3$, by \eqref{LES:gammatoe} and \eqref{equation:gammatoe} we see that $\gamma_{\ell} = 0$ for
\[ \ell > a_{1} + \cdots + a_{e-1} + 2a_{e}/3 - e - 1. \]

Note that $\beta_{i} = \sum_{v=0}^{i} \gamma_{v}$ so
\[ \sum_{i=0}^{\ell} \beta_{i} = \sum_{i=0}^{\ell} \sum_{v=0}^{i} \gamma_{v} = \sum_{i=0}^{\ell} (\ell - i + 1) \gamma_{i}. \]
From the ideal sheaf sequence for $\Gamma \subset Y$ and the fact that $H^{1}(Y, \I_{\Gamma}(\ell)) = 0$ for $\ell \gg 0$, it follows that
\[ \sum_{i=0}^{\ell} \gamma_{i} = \beta_{\ell} = d, \qquad \ell \gg 0. \]
We may substitute this into the expression above to get
\[ \sum_{i=0}^{\ell} \beta_{i} = \ell d - \sum_{i=0}^{\ell} (i - 1) \gamma_{i}, \qquad \ell \gg 0. \]
On the other hand, for $\ell \gg 0$ we have $\alpha_{\ell} \geq \sum_{i=0}^{\ell} \beta_{i}$. From the definition of $\alpha_{\ell}$ and Riemann-Roch, it follows that for $\ell \gg 0$:
\[ p_{a}(C) = d \ell - \alpha_{\ell} + 1 \leq \sum_{i=0}^{\ell} (i-1)\gamma_{i} + 1. \]
So we would like to produce a function $\gamma_{i}^{\text{max}}$ which maximizes the sum above, subject to the constraint $\sum_{i=1}^{\ell} \gamma_{i} = d$ for $\ell \gg 0$.

Using the fact that $\gamma_{\ell} = 0$ for $\ell > a_{1} + \cdots + a_{e-1} + 2a_{e}/3 - e - 1$, we can write down a simple example of a function $\gamma_{i}^{\text{max}}$ that produces a rough upper bound for the arithmetic genus:
\[
\gamma_{\ell}^{\text{max}} =
\begin{cases} \deg C & \text{ if } \ell = a_{1} + \cdots + a_{e-1} + \lfloor 2a_{e}/3 \rfloor - e - 1, \\ 0 & \text{ otherwise}. \end{cases}
\]
We then obtain
\begin{align*}
p_{a}(C) &\leq 1 + \sum_{i=0}^{\ell} (i-1) \gamma_{i}^{\text{max}} \\
&= 1 + (a_{1} + \cdots + a_{e-1} + \lfloor 2a_{e}/3 \rfloor - e - 2) \deg C \\
&\leq (a_{1} + \cdots + a_{e-1} + 2a_{e}/3) \deg C \qedhere
\end{align*}

\end{proof}

Now we will use the induction hypothesis, which says that the divisor
\[ L_{I} := a_{e}H - \sum_{i \in I} 3E_{i} \]
is nef for any subset $I \subset \{ 1, 2, \ldots, s+1 \}$ with $\# I = s$. Averaging over all $I$ shows that
\[ L_{s+1} := \frac{s+1}{s} a_{e}H_{Y} - \sum_{i=1}^{s+1} 3 E_{i} \]
is nef. This implies that $L_{s+1} \cdot \widetilde{C} \geq 0$, and hence
\begin{equation}\label{equation:miInductionSum}
\sum_{i=1}^{s+1} m_{i} \leq \frac{s+1}{3s} a_{e}k.
\end{equation}

We can now give an upper bound for
\[ \frac{2^{3/2}}{3} \sum_{i=1}^{s+1} m_{i}^{3/2} \]
by integrating the various estimates that have been established so far. Applying Proposition~\ref{prop:multiplicityDeltaInvariant} to $p_{i} \in C \subset Y$ ($1 \leq i \leq s+1$) with $n = 3$ gives:
\[ p_{a}(C) - p_{g}(C) \geq \sum_{i=1}^{s+1} \frac{2^{3/2}}{3} \cdot m_{i}^{3/2} - 2m_{i}. \]
This may be rewritten as
\begin{align*}
\sum_{i=1}^{s+1} \frac{2^{3/2}}{3} \cdot m_{i}^{3/2} &\leq p_{a}(C) - p_{g}(C) + 2 \sum_{i=1}^{s+1} m_{i} \\
&\leq p_{a}(C) + a_{e}k \\
&\leq (a_{1} + \cdots + a_{e-1} + \frac{5}{3} a_{e}) k,
\end{align*}
where the second step follows from $s \geq 2$ and \eqref{equation:miInductionSum}. Next, combine this with the lower bound coming from \eqref{lowerBoundm3/2}:
\[ \frac{2^{3/2}}{3^{5/2} \cdot (s+1)^{1/2}} k^{3/2} a_{e}^{3/2} < \sum_{i=1}^{s+1} \frac{2^{3/2}}{3} \cdot m_{i}^{3/2} \leq (a_{1} + \cdots + a_{e-1} + \frac{5}{3} a_{e}) k. \]
We can square both sides, solve for $k$, and simplify using the inequalities $a_{1} \leq \cdots \leq a_{e}$:
\begin{align*}
k &< \frac{3^{5}}{2^{3}} \cdot \frac{(a_{1} + \cdots + a_{e-1} + 5a_{e}/3)^{2}}{a_{e}^{3}} (s+1) \\
&\leq \frac{3^{5}}{2^{3}} \cdot \frac{2^{2} [(e+2/3)a_{e}]^{2}}{a_{e}^{3}} (s+1) \\
&= \frac{3^{3}(3e+2)^{2}}{2 a_{e}} (s+1) \\
&\leq \frac{(3e+2)}{3 \cdot ((e+1)!)^{e}} \cdot \alpha < \frac{1}{e! ((e+1)!)^{e-1}} \alpha.
\end{align*}

On the other hand, recall from Set-up~\ref{setup:ciSurfacenef} that the integers $a_{i}$ are of the form $a_{i} = (e+1)! \cdot q_{i}$ for some integers $q_{i}$ which are pairwise coprime ($i = 1, \ldots, e-1$). By applying Proposition~\ref{degreeCurves} to $C \subset Y$ with $n = 3$ and $f = e-1$, it follows that
\[ k \geq \frac{1}{e!} \cdot q_{1} \cdots q_{e-1} = \frac{1}{e! ((e+1)!)^{e-1}} \alpha, \]
which is a contradiction. This completes the proof of Theorem~\ref{thm:dim2nef}.

\begin{remark}\label{remark:primes}
Lastly, we would like to explain how to choose the numbers $a_{i}$ with the desired properties in the proof of Theorem~\ref{thm:cgCIsurfaces}. One method is to invoke a theorem of Sondow \cite{Sondow09}, which can be viewed as an extension of Bertrand's postulate. Let $\pi(x)$ denote the number of primes not exceeding $x$, and recall:

\begin{definition}
For $n \geq 1$, the $n$th \textit{Ramanujan prime} is the smallest positive integer $R_{n}$ such that if $x \geq R_{n}$, then $\pi(x) - \pi(\frac{1}{2}x) \geq n$. 
\end{definition}

\noindent By \cite[Theorem 2]{Sondow09}, the $n$th Ramanujan prime satisfies the inequalities
\[ 2n \log 2n < R_{n} < 4n \log 4n, \qquad (n \geq 1). \]
For our purposes, the upper bound will suffice. Set $n = e-1$ and let $d_{1} \leq \cdots \leq d_{e-1}$ be integers such that
\[ d_{1} \geq (e+1)! \cdot 4(e-1) \log 4(e-1). \]
Then this lower bound holds for all $d_{i}$ so Sondow's theorem tells us that there are always $e-1$ primes between $d_{i}/(2 \cdot (e+1)!)$ and $d_{i}/(e+1)!$ for $i = 1, \ldots, e-1$. Working in reverse order $i = e-1, e-2, \ldots, 1$, it follows that we can always choose \textit{distinct} primes $q_{i}$ such that
\[ \frac{1}{2} \cdot \frac{d_{i}}{(e+1)!} < q_{i} \leq \frac{d_{i}}{(e+1)!}. \]
Finally, we may define $a_{i} \colonequals (e+1)! \cdot q_{i}$.
\end{remark}

%
%

\bibliographystyle{siam}
\bibliography{Biblio}

\footnotesize{
\textsc{Department of Mathematics, Harvard University, Cambridge, Massachusetts 02138}

\textit{E-mail address:} \href{mailto:nathanchen@math.harvard.edu}{nathanchen@math.harvard.edu}
}

\end{document}